\documentclass[leqno,12pt]{article} 
\usepackage{amsmath, amssymb}
\usepackage{amsthm} 
\theoremstyle{plain} 
\newtheorem{theorem}{\indent\sc Theorem}[section]
\newtheorem{lemma}[theorem]{\indent\sc Lemma}
\newtheorem{corollary}[theorem]{\indent\sc Corollary}
\newtheorem{proposition}[theorem]{\indent\sc Proposition}

\theoremstyle{definition} 
\newtheorem{definition}[theorem]{\indent\sc Definition}
\newtheorem{remark}[theorem]{\indent\sc Remark}


\title{Harmonic metallic structures}

\author{Adara M. Blaga and Antonella Nannicini}

\date{}

\begin{document}

\maketitle

\markboth{{\small\it {\hspace{4cm} Harmonic metallic structures}}}{\small\it{Harmonic metallic structures
\hspace{4cm}}}


\begin{abstract}
The concept of harmonic metallic structure on a metallic pseudo-Riemannian manifold is introduced. In the case of compact manifolds we prove that harmonicity of a metallic structure $J$,  with $J^2=pJ+qI$ and $p^2+4q\neq 0$, is  equivalent to $dJ=0$.
Conditions for a harmonic metallic structure to be preserved by harmonic maps are also given.
Moreover, we consider harmonic metallic structures on the generalized tangent bundle, provide a Weitzenb\"{o}ck formula for the dual metallic structure and express the Hodge-Laplace operator on $TM \oplus T^*M$.
\end{abstract}

\bigskip

\section{Introduction}

Inspired by the paper of W. Jianming \cite{ji}, we introduce the notion of \textit{harmonic metallic structure} and underline the connection between harmonic metallic structures and harmonic maps.  It is well known that harmonic maps play an important role in many areas of mathematics. They often appear in nonlinear theories because of the nonlinear nature of the corresponding partial differential equations. In theoretical physics, harmonic maps are also known as sigma models. Remark also that harmonic maps between manifolds endowed with different geometrical structures have been studied in many contexts: S. Ianu\c s and A. M. Pastore treated the case of contact metric manifolds \cite{pa}, C.-L. Bejan and M. Benyounes the almost para-Hermitian manifolds \cite{bej}, B. Sahin the locally conformal K\"{a}hler manifolds \cite{sa}, S. Ianu\c s, R. Mazzocco and G. E. V\^ ilcu the quaternionic K\"{a}hler manifolds \cite{i}, J. P. Jaiswal the Sasakian manifolds \cite{ja}, D. Fetcu the complex Sasakian manifolds \cite{fe}, J. Li the Finsler manifolds \cite{li} etc. A. Fotiadis studied the noncompact case, describing the problem of finding a harmonic map between noncompact manifolds \cite{fo}.

In the present paper, first, we consider the case of compact Riemannian manifolds and prove that harmonicity of a metallic structure $J$ with $J^2=pJ+qI$ and $p^2+4q\neq 0$, is equivalent to $dJ=0$, then we relate harmonicity to integrability. Conditions for a harmonic metallic structure to be preserved by harmonic maps are also given.
Moreover, we consider harmonic metallic structures on the generalized tangent bundle, provide a Weitzenb\"{o}ck formula for the dual metallic structure and express the Hodge-Laplace operator on $TM \oplus T^*M$.

\section{Harmonic metallic structures}

\subsection{Preliminaries}

Let $(M,g)$ be an $n$-dimensional pseudo-Riemannian manifold. We recall that a metallic pseudo-Riemannian structure $J$ on $M$ is a $g$-symmetric $(1,1)$-tensor field on $M$ such that $J^2=pJ+qI$, for some $p$ and $q$ real numbers, and $(M,J,g)$ is called a \textit{metallic pseudo-Riemannian manifold} (\cite{bn}, \cite{bn1}, \cite{bn2}).

Let $\nabla$ be the Levi-Civita connection associated to $g$. Consider the exterior differential and
codifferential operators defined for any tangent bundle-valued
$p$-form $T\in \Gamma(\Lambda^pT^*M\otimes TM)$ by
$$(dT)(X_1,\dots,X_{p+1}):=-\sum_{i=1}^{p+1}(-1)^{i}(\nabla_{X_i}T)(X_1,\dots,\widehat{X_i},\dots,X_{p+1})$$
and
$$(\delta T)(X_1,\dots,X_{p-1}):=-\sum_{i=1}^{n}(\nabla_{E_i}T)(E_i,X_1,\dots,X_{p-1}),$$
for $\{E_i\}_{1\leq i\leq n}$ a $g$-orthonormal frame field, and the
Hodge-Laplace operator on $\Gamma(\Lambda^pT^*M\otimes TM)$ by
$$
\Delta:=d\circ \delta+\delta \circ d.
$$

W. Jianming studied in \cite{ji} some properties of harmonic complex
structures and we discussed in \cite{ada} the almost tangent case and in \cite{ablaga} the para-cosymplectic case.

We pose the following:

\begin{definition}
A metallic structure $J$ is called \textit{harmonic} if $\Delta J=0$.
\end{definition}

If $M$ is compact and $g$ is positive definite, from the definition it follows that $J$ is harmonic if and only if
$dJ=0$ and $\delta J=0$ which is equivalent to
$(\nabla_XJ)Y=(\nabla_YJ)X$, for any $X$, $Y\in C^{\infty}(TM)$ and
$trace(\nabla J)=0$.

\subsection{Properties of harmonic metallic structures}

Let $(M,J,g)$ be an $n$-dimensional metallic pseudo-Riemannian manifold, let $\nabla$ be the Levi-Civita connection associated to $g$ and let $\{E_i\}_{1\leq i\leq n}$ be a $g$-orthonormal frame field.

We have the followings:

\begin{lemma}  
$$\sum_{i=1}^{n}g((dJ)(X,E_i),E_i)=\sum_{i=1}^{n}g((\nabla_X J)E_i,E_i)+g(X,\delta J),$$
for any $X\in C^{\infty}(TM)$.
\end{lemma}
\begin{proof} We have:
$$\sum_{i=1}^{n}g((dJ)(X,E_i),E_i)=\sum_{i=1}^{n}g((\nabla_X J)E_i-{\nabla}_{E_i}JX+J({\nabla}_{E_i}X),E_i)=$$
$$=\sum_{i=1}^{n}[g((\nabla_X J)E_i,E_i)-E_i(g(JX,E_i))+g(JX,\nabla_{E_i}E_i)+g(\nabla_{E_i}X,JE_i)]=$$
$$=\sum_{i=1}^{n}[g((\nabla_X J)E_i,E_i)-E_i(g(X,JE_i))+g(X,J(\nabla_{E_i}E_i))+g(\nabla_{E_i}X,JE_i)]=$$
$$=\sum_{i=1}^{n}[g((\nabla_X J)E_i,E_i)-g(X,\nabla_{E_i}JE_i)+g(X,J(\nabla_{E_i}E_i))]=$$
$$=\sum_{i=1}^{n}[g((\nabla_X J)E_i,E_i)-g(X,(\nabla_{E_i}J)E_i)]=\sum_{i=1}^{n}g((\nabla_X J)E_i,E_i)+g(X,\delta J).$$
\end{proof}

\begin{corollary} $$dJ=0 \Longrightarrow g(X,\delta J)=-\sum_{i=1}^{n}g((\nabla_X J)E_i,E_i), \ \ \textit{for any} \ \ X\in C^{\infty}(TM).$$
\end{corollary}

\begin{lemma}
$$\sum_{i=1}^{n}g((dJ)(X,E_i),JE_i)={1\over 2}p\sum_{i=1}^{n}g((\nabla_X J)E_i,E_i)+pg(X,\delta J)-g(JX,\delta J),$$
for any $X\in C^{\infty}(TM)$.
\end{lemma}
\begin{proof} $$\sum_{i=1}^{n}g((dJ)(X,E_i),JE_i)=\sum_{i=1}^{n}[g((\nabla_X J)E_i,JE_i)-g((\nabla_{E_i}J)X,JE_i)]=$$
$$=\sum_{i=1}^{n}[g(\nabla_X JE_i,JE_i)-g(J(\nabla_XE_i),JE_i)-g(\nabla_{E_i}JX,JE_i)+g(J(\nabla_{E_i}X),JE_i)]=$$
$$=\sum_{i=1}^{n}[{1\over 2}X(g(JE_i,JE_i))-pg(\nabla_XE_i,JE_i)-qg(\nabla_X E_i,E_i)-E_i(g(JX,JE_i))+$$
$$+g(JX,\nabla_{E_i}JE_i)+pg(\nabla_{E_i}X,JE_i)+qg(\nabla_{E_i}X,E_i)]=$$
$$=\sum_{i=1}^{n}[{1\over 2}pX(g(E_i,JE_i))-pX(g(E_i,JE_i))+pg(\nabla_XJE_i,E_i)+$$
$$+pg(X,J(\nabla_{E_i}E_i))-pg(X,\nabla_{E_i}JE_i)]-g(JX,\delta J)=$$
$$=\sum_{i=1}^{n}[-{1\over 2}pg(\nabla_XE_i,JE_i)+{1\over 2}pg(E_i,\nabla_X JE_i)]+pg(X,\delta J)-g(JX,\delta J)=$$
$$=\sum_{i=1}^{n}{1\over 2}pg((\nabla_X J)E_i,E_i)+pg(X,\delta J)-g(JX,\delta J).$$
\end{proof}

As an application, we get the following:

\begin{proposition} Let $(M,J,g)$ be a metallic pseudo-Riemannian manifold such that $J^2=pJ+qI$ with $p^2+4q\neq 0$. Then $dJ=0$ implies $\delta J=0$.
\end{proposition}
\begin{proof} If $dJ=0$, then:
$$\sum_{i=1}^{n}g((\nabla_X J)E_i,E_i)=-g(X,\delta J)$$
and
$$\sum_{i=1}^{n}{1\over 2}pg((\nabla_X J)E_i,E_i)=-pg(X,\delta J)+g(JX,\delta J).$$
Hence
$$g(JX-{1\over 2}pX,\delta J)=0,$$
for any $X \in C^{\infty}(TM)$. If $p^2+4q\neq 0$, then ${1\over 2}p$ is not an eigenvalue of $J$ and so $J-{1\over 2}pI$ is invertible. In particular, $g(Y,\delta J)=0$, for any $Y \in C^{\infty}(TM)$ and this implies $\delta J=0$.
\end{proof}

Furthermore:

\begin{corollary} Let $(M,J,g)$ be a compact metallic Riemannian manifold such that $J^2=pJ+qI$ with $p^2+4q\neq 0$. Then $J$ is harmonic if and only if
$dJ=0$.
\end{corollary}
\begin{proof} Let $J$ be harmonic. Since $M$ is compact and $g$ is positive definite, $\Delta J=0$ implies $dJ=0$. Conversely, if $dJ=0$, then also $\delta J=0$ and furthrmore $\Delta J=0$. Then we get the statement.
\end{proof}

The vanishing if $dJ$ is also related to integrability, namely we have the following:

\begin{lemma} 
$$(dJ)(JX,Y)+(dJ)(X,JY)-p(dJ)(X,Y)=N_J(X,Y),$$
for any $X$, $Y\in C^{\infty}(TM)$, where $N_J$ is the Nijenhuis tensor of $J$.
\end{lemma}
\begin{proof}
Let $X$, $Y\in C^{\infty}(TM)$. Then
$$
(dJ)(X,Y)=(\nabla_X J)Y-(\nabla_Y
J)X=[X,JY]+\nabla_{JY}X-[Y,JX]-\nabla_{JX}Y-J([X,Y]);
$$
$$
(dJ)(JX,Y)=[JX,JY]+\nabla_{JY}JX-[Y,J^2 X]-\nabla_{J^2X}Y-J([JX,Y])=$$
$$=[JX,JY]+\nabla_{JY}JX-p[Y,J X]-q[Y,X]-p\nabla_{JX}Y-q\nabla_X Y-J([JX,Y]);$$
$$(dJ)(X,JY)=[X,J^2Y]+\nabla_{J^2Y}X-[JY,JX]-\nabla_{JX}JY-J([X,JY])=
$$$$=p[X,JY]+q[X,Y]+p\nabla_{JY}X+q\nabla_Y X-[JY,JX]-\nabla_{JX}JY-J([X,JY]).
$$
Hence:
$$(dJ)(JX,Y)+(dJ)(X,JY)-p(dJ)(X,Y)=[JX,JY]-J([JX,Y])-J([X,JY])+J^2([X,Y])$$
and the proof is complete.
\end{proof}

As a consequence, we get the following:
\begin{proposition}
Let $(M,J,g)$ be a compact metallic Riemannian manifold. If $J$ is harmonic, then it is integrable.
\end{proposition}
\begin{proof}
Since $M$ is compact and $g$ is positive definite, $\Delta J=0$ implies $dJ=0$, then we get the statement.
\end{proof}

Moreover:

\begin{proposition} Let $(M,J,g)$ be a compact metallic Riemannian manifold. Then $M$ is locally metallic if and only if $J$ is harmonic and $M$ is nearly K\"ahler manifold.
\end{proposition}
\begin{proof} If $M$ is locally metallic, i.e. $\nabla J=0$, then $J$ is harmonic and $(M,J,g)$ is nearly K\"ahler.

Conversely, if $(M,J,g)$ is nearly K\"ahler, then $(\nabla_X J)Y=-(\nabla_Y
J)X$ and from $dJ=0$ we get $\nabla J=0$.
\end{proof}

From $\nabla I=0$ we have $\Delta I=0$ and we get the followings:

\begin{proposition} Let $(M,J,g)$ be a metallic pseudo-Riemannian manifold such that $J^2=pJ+qI$ with $p^2+4q>0$ and let $J_p={{1}\over{\sqrt{p^2+4q}}}(2J-pI)$ be the almost product structure on $M$ associated to $J$. Then $J$ is harmonic if and only if $J_p$ is harmonic.
\end{proposition}
\begin{proof}
$\Delta J_p={{2}\over{\sqrt{p^2+4q}}}\Delta J$ and the proof is complete.
\end{proof}

\begin{proposition} Let $(M,J,g)$ be a metallic pseudo-Riemannian manifold such that $J^2=pJ+qI$ with $p^2+4q<0$ and let $J_c={{1}\over{\sqrt{-p^2-4q}}}(2J-pI)$ be the Norden structure on $M$ associated to $J$. Then $J$ is harmonic if and only if $J_c$ is harmonic.
\end{proposition}
\begin{proof}
$\Delta J_c={{2}\over{\sqrt{-p^2-4q}}}\Delta J$ and the proof is complete.
\end{proof}

\subsection{Bochner formula}

We know that for any tangent bundle-valued differential form, $T\in \Gamma(\Lambda^1T^*M\otimes TM)$, the following Weitzenb\"{o}ck formula holds \cite{xi}:
$$
\Delta T=-\nabla^2T-S,
$$
where
$\nabla^2T:=\sum_{i=1}^{n}(\nabla_{E_i}\nabla_{E_i}T-\nabla_{\nabla_{E_i}E_i}T)$
and $SX:=\sum_{i=1}^{n}(R(E_i,X)T)E_i$, $X\in C^{\infty}(TM)$, for
$\{E_i\}_{1\leq i\leq n}$ a $g$-orthonormal frame field and
$R(X,Y):=\nabla_X\nabla_Y-\nabla_Y\nabla_X-\nabla_{[X,Y]}$, $X$,
$Y\in C^{\infty}(TM)$, the Riemann curvature tensor field. We shall also
use the notation $R(X,Y,Z,W)=:g(R(X,Y)Z,W)$,
$X$, $Y$, $Z$, $W\in C^{\infty}(TM)$. \\
On the metallic pseudo-Riemannian manifold $(M,J,g)$, taking $T$ equal to $J$, for
any vector field $X$, we have
$$SX:=\sum_{i=1}^{n}(R(E_i,X)J)E_i=\sum_{i=1}^{n}[R(E_i,X)JE_i-J(R(E_i,X)E_i)].$$

We can state:

\begin{proposition}
Let $(M,J,g)$ be an $n$-dimensional metallic pseudo-Riemannian manifold. If $J$ is harmonic, then
$$
|\nabla{J}|^2=\sum_{1\leq
i,j \leq n}R(E_i,E_j,JE_i,J
E_j)+p\cdot trace (J\circ Q)-q \cdot scal,
$$
for $\{E_i\}_{1\leq i\leq n}$ a $g$-orthonormal frame field
in a neighborhood of a point $x\in M$ such that $({{\nabla}_{E_i} E_j})(x)=0$, $1\leq i,j\leq n$, $Q$ the Ricci operator defined by $g(QX,Y):=Ric(X,Y)$ and $scal$ the scalar curvature of $(M,g)$.
\end{proposition}
\begin{proof}
A similar computation like in \cite{ji} leads us to
$$
\langle \nabla^2J,J\rangle=
\sum_{i=1}^{n}\langle\nabla_{E_i}\nabla_{E_i}J,J\rangle=
-|\nabla{J}|^2
$$
and
$$
\langle S,J\rangle=\sum_{j=1}^{n}g(
SE_j,JE_j)= \sum_{1\leq
i,j \leq n}g(R(E_i,E_j)JE_i,J
E_j)-\sum_{1\leq
i,j \leq n}g(J(R(E_i,E_j)E_i),J
E_j)=$$$$=\sum_{1\leq
i,j \leq n}R(E_i,E_j,JE_i,J
E_j)-\sum_{1\leq
i,j \leq n}R(E_i,E_j,E_i,J^2
E_j)=$$$$=\sum_{1\leq
i,j \leq n}R(E_i,E_j,JE_i,J
E_j)+p\sum_{j=1}^nRic(E_j,JE_j)-q \cdot scal=$$$$=\sum_{1\leq
i,j \leq n}R(E_i,E_j,JE_i,J
E_j)+p\cdot trace (J\circ Q)-q \cdot scal.
$$

Therefore
$$
0=\langle \Delta J,
J\rangle=-\langle \nabla^2
J,J\rangle-\langle S,J\rangle=$$$$=
|\nabla{J}|^2-\sum_{1\leq
i,j \leq n}R(E_i,E_j,JE_i,J
E_j)-p\cdot trace (J\circ Q)+q \cdot scal.
$$
\end{proof}

\begin{remark}
If $(M,J,g)$ is a locally metallic pseudo-Riemannian manifold, then $\langle \Delta J,J\rangle=0$.
\end{remark}

\section{Harmonic maps and harmonic metallic structures}

Let $(M,J,g)$ and $(\bar{M},\bar{g},\bar{J})$ be two $n$-dimensional metallic Riemannian manifolds. Denote by $\nabla$ and respectively, $\bar{\nabla}$ the Levi-Civita connections associated to $g$ and respectively, $\bar{g}$.

Consider $\Phi:(M,J,g)\rightarrow (\bar{M},\bar{J},\bar{g})$ a smooth map and let $$\tau(\Phi):=\sum_{i=1}^{n}[\bar{\nabla}_{\Phi_*E_i}\Phi_*E_i-\Phi_*(\nabla_{E_i}E_i)]$$ be \textit{the tension field} of $\Phi$, where $\{E_i\}_{1\leq i\leq n}$ is a $g$-orthonormal frame field
on $TM$.

\begin{proposition}\label{p}
Let $\Phi:(M,J,g)\rightarrow (\bar{M},\bar{J},\bar{g})$ be a metallic isometry. Then
$$
\bar{J}(\tau(\Phi))+\Phi_*(\delta J)-\delta \bar{J}=
\sum_{i=1}^{n}[\bar{\nabla}_{\Phi_*E_i}\Phi_*(JE_i)-\Phi_*(\nabla_{E_i}JE_i)],
$$
for $\{E_i\}_{1\leq i\leq n}$ a $g$-orthonormal frame field
on $TM$.
\end{proposition}
\begin{proof}
Express $\delta J=-\sum_{i=1}^{n}(\nabla_{E_i} J)E_i=-\sum_{i=1}^{n}[\nabla_{E_i}JE_i-J(\nabla_{E_i}E_i)]$ and replace it in the left-side of the relation.
\end{proof}

\begin{corollary}
Let $\Phi:(M,g,J)\rightarrow (\bar{M},\bar{g},\bar{J})$ be a metallic isometry. If there exists a $g$-orthonormal frame field $\{E_i\}_{1\leq i\leq n}$
on $TM$ such that $\Phi_*((\nabla_{E_i}J)E_i)=(\bar{\nabla}_{\Phi_*{E_i}}\bar{J})(\Phi_*E_i)$, then
$$
\delta \bar{J}=
\Phi_*(\delta J).
$$
\end{corollary}
\begin{proof}
We have $\bar{\nabla}_{\Phi_*E_i}\bar{J}(\Phi_*E_i)-\Phi_*(\nabla_{E_i}JE_i)=\bar{J}(\bar{\nabla}_{\Phi_*E_i}\Phi_*E_i)-\Phi_*(J(\nabla_{E_i}E_i))$ and we get
$$\bar{J}(\tau(\Phi))+\Phi_*(\delta J)- \delta \bar{J}=\bar{J}(\sum_{i=1}^{n}[\bar{\nabla}_{\Phi_*E_i}\Phi_*E_i-\Phi_*(\nabla_{E_i}E_i)])=\bar{J}(\tau(\Phi)).$$
\end{proof}

\begin{definition}
A smooth map $\Phi:(M,J,g)\rightarrow (\bar{M},\bar{J},\bar{g})$ is said to be \textit{harmonic} if its tension field
$\tau(\Phi)$ vanishes.
\end{definition}

\begin{proposition}
Let $\Phi:(M,J,g)\rightarrow (\bar{M},\bar{J},\bar{g})$ be a metallic isometry. If $\Phi$ is a harmonic map, then
$$
\Phi_*(\delta J)=\delta \bar{J}+
\sum_{i=1}^{n}[\bar{\nabla}_{\Phi_*E_i}\Phi_*(JE_i)-\Phi_*(\nabla_{E_i}JE_i)],
$$
for $\{E_i\}_{1\leq i\leq n}$ a $g$-orthonormal frame field
on $TM$.

Moreover, if there exists a $g$-orthonormal frame field $\{E_i\}_{1\leq i\leq n}$
such that $\Phi_*((\nabla_{E_i}J)E_i)=(\bar{\nabla}_{\Phi_*{E_i}}\bar{J})(\Phi_*E_i)$, then
$$
\sum_{i=1}^{n}[\bar{\nabla}_{\Phi_*E_i}\Phi_*(JE_i)-\Phi_*(\nabla_{E_i}JE_i)]=0.
$$
\end{proposition}

\begin{corollary}
Let $\Phi:(M,g,J)\rightarrow (\bar{M},\bar{g},\bar{J})$ be a metallic isometry and assume that $J$ is a harmonic metallic structure.
\begin{enumerate}
  \item If there exists a $g$-orthonormal frame field $\{E_i\}_{1\leq i\leq n}$
on $TM$ such that $\Phi_*((\nabla_{E_i}J)E_i)=(\bar{\nabla}_{\Phi_*{E_i}}\bar{J})(\Phi_*E_i)$, then
$\delta \bar{J}=0$, hence $\bar{J}$ is a harmonic metallic structure, too.
  \item If $\Phi$ is a harmonic map, then
$$\delta \bar{J}=
-\sum_{i=1}^{n}[\bar{\nabla}_{\Phi_*E_i}\Phi_*(JE_i)-\Phi_*(\nabla_{E_i}JE_i)],$$
for $\{E_i\}_{1\leq i\leq n}$ a $g$-orthonormal frame field
on $TM$.
\end{enumerate}
\end{corollary}

\begin{remark}
If $\Phi:(M,g,J)\rightarrow (\bar{M},\bar{g},\bar{J})$ is a metallic isometry, then either $p=\bar{p}$ and $q=\bar{q}$ or $J$ and $\bar{J}$ are trivial metallic structures, namely equal to $\frac{\bar{q}-q}{p-\bar{p}}I$, for $p\neq \bar{p}$.
Indeed, for any $X$, $Y\in C^{\infty}(TM)$, we have:
$$pg(JX,Y)+qg(X,Y)=g(JX,JY)=\bar{g}(\Phi_*(JX),\Phi_*(JY))=\bar{g}(\bar{J}(\Phi_*X),\bar{J}(\Phi_*Y))=$$$$=
\bar{p}\bar{g}(\bar{J}(\Phi_*X),\Phi_*Y)+\bar{q}\bar{g}(\Phi_*X,\Phi_*Y)=\bar{p}\bar{g}(\Phi_*(JX),\Phi_*Y)+\bar{q}\bar{g}(\Phi_*X,\Phi_*Y)=$$$$=
\bar{p}g(JX,Y)+\bar{q}g(X,Y),$$
which implies $(p-\bar{p})J=(\bar{q}-q)I$ and similarly, $(p-\bar{p})\bar{J}=(\bar{q}-q)I$.
\end{remark}

\section{Harmonic generalized metallic structures}

\subsection{Metallic structures on the generalized tangent bundle}

Let $TM\oplus T^*M$ be the generalized tangent bundle of $M$ and let $(\hat J, \hat g)$ be the generalized metallic pseudo-Riemannian structure induced by $(J,g)$ \cite{bn}. In block matrix form, $\hat{J}$ is written as
$$\hat{J}=\begin{pmatrix}
               J & 0 \\
               \flat_g & -J^*+pI \\
         \end{pmatrix}$$
and the metric $\hat{g}$ is given by
$$
\hat{g}(X+\alpha,Y+\beta)=g(X,Y)+g(\sharp_g\alpha,\sharp_g\beta)+{1\over{p^2+4q}}[p(\alpha(Y)+\beta(X))-2(\alpha(JY)+\beta(JX))]=$$
$$=g(X,Y)+g(\sharp_g\alpha,\sharp_g\beta)+{{{\sqrt{\vert p^2+4q\vert}}}\over{p^2+4q}}[\alpha(\tilde J Y)+\beta(\tilde J X)],$$
for any $X$, $Y\in C^{\infty}(TM)$ and $\alpha$, $\beta\in C^{\infty}(T^*M)$, where $\flat_g$ and $\sharp_g$ are the musical isomorphisms induced by $g$.

Let $\nabla$ be the Levi-Civita connection of $g$ and let $\hat \nabla$ be the induced generalized connection \cite{bn}:
$$\hat \nabla :C^{\infty}(TM \oplus T^*M)\times C^{\infty}(TM \oplus T^*M) \rightarrow C^{\infty}(TM \oplus T^*M)$$
$${\hat \nabla}_{X+\alpha}(Y+\beta):={\nabla}_X Y+{\nabla}_X \beta,$$
for any $X$, $Y\in C^{\infty}(TM)$ and $\alpha$, $\beta\in C^{\infty}(T^*M)$.

\subsection{Harmonicity of generalized metallic structures}

Define the exterior differential and
codifferential operators for any $TM\oplus T^*M$-valued $p$-form $T\in \Gamma(\Lambda^p (TM\oplus T^*M)^*\otimes (TM\oplus T^*M))$ by
$$(dT)(\sigma_1,\dots,\sigma_{p+1}):=-\sum_{i=1}^{p+1}(-1)^{i}(\hat \nabla_{\sigma_i}T)(\sigma_1,\dots,\widehat{\sigma_i},\dots,\sigma_{p+1})$$
and
$$(\delta T)(\sigma_1,\dots,\sigma_{p-1}):=-\sum_{i=1}^{n}(\hat \nabla_{\xi_i}T)(\xi_i,\sigma_1,\dots,\sigma_{p-1}),$$
for $\{\xi_i\}_{1\leq i\leq n}$ a $\hat{g}$-orthonormal frame field, and the
Hodge-Laplace operator on $\Gamma(\Lambda^p(TM\oplus T^*M)^*\otimes (TM\oplus T^*M))$ by
$$
\Delta:=d\circ \delta+\delta \circ d.
$$

\begin{proposition} \label{p1}
$$(d \hat J) (X+\alpha,Y+\beta)=(dJ)(X,Y)+(\nabla_YJ^*)\alpha-(\nabla_XJ^*)\beta,$$
for any $X$, $Y\in C^{\infty}(TM)$ and $\alpha$, $\beta\in C^{\infty}(T^*M)$.
\end{proposition}
\begin{proof} A direct computation gives:
$$(d \hat J) (X+\alpha,Y+\beta)=({\hat \nabla}_{X+\alpha}\hat J)(Y+\beta)-({\hat \nabla}_{Y+\beta}\hat J)(X+\alpha)=$$
$$={\hat \nabla}_{X+\alpha}\hat J(Y+\beta)-\hat J({\hat \nabla}_{X+\alpha}(Y+\beta))
-{\hat \nabla}_{Y+\beta}\hat J(X+\alpha)+\hat J({\hat \nabla}_{Y+\beta}(X+\alpha))=$$
$$={\nabla}_XJY+{\nabla}_X\flat_g(Y)-{\nabla}_XJ^*\beta+p{\nabla}_X \beta-J{\nabla}_XY-\flat_g({\nabla}_XY)+J^*({\nabla}_X\beta)-p{\nabla}_X\beta-$$
$$-{\nabla}_YJX-{\nabla}_Y\flat_g(X)+{\nabla}_YJ^*\alpha-p{\nabla}_Y \alpha+J{\nabla}_YX+\flat_g({\nabla}_YX)-J^*({\nabla}_Y\alpha)+p{\nabla}_Y\alpha=$$
$$=({\nabla}_XJ)Y-({\nabla}_YJ)X+({\nabla}_YJ^*)\alpha-({\nabla}_XJ^*)\beta=$$
$$=(dJ)(X,Y)+(\nabla_YJ^*)\alpha-(\nabla_XJ^*)\beta.$$
\end{proof}

\begin{proposition} $d\hat J=0$ if and only if $(M,J,g)$ is locally metallic.
\end{proposition}
\begin{proof} From Proposition \ref{p1} it follows that  $d\hat J=0$ if and only if $dJ=0$ and $\nabla J^*=0$. Since $\nabla J^*=0$ if and only if $\nabla J=0$ and moreover, $\nabla J=0$ implies $dJ=0$, we get the statement.
\end{proof}

In order to compute the codifferential operator we need a $\hat g$-orthonormal basis of $TM\oplus T^*M$.

First we recall the following.
Let $(M,J,g)$ be a metallic pseudo-Riemannian manifold such that $J^2=pJ+qI$ with $p^2+4q\neq 0$. Define
$$\tilde J:={-1\over{\sqrt{\vert p^2+4q\vert}}}(2J-pI).$$
A direct computation gives the following:

\begin{lemma} $${\tilde J}^2={p^2+4q\over{{\vert p^2+4q\vert}}}I.$$
In particular, $\tilde J$ is an almost product structure for $p^2+4q>0$ and an almost complex structure for $p^2+4q<0$.
\end{lemma}

We shall consider the case when $g$ is a Riemannian metric, which, in particular, implies that $p^2+4q>0$ (because if $g$ is positive definite, then $\tilde J$ has necessarily real eigenvalues, and ${\tilde J}^2=I$).

Let $n$ be the real dimension of $M$ and let $\{E_i\}_{1\leq i\leq n}$ be a $g$-orthonormal frame field for $TM$. If we
consider $\{\flat_g(E_i)\}_{1\leq i\leq n}$ and ${\sigma}_i:={\tilde J} E_i+\flat_g(E_i)$, we have:
$$\hat g ({\sigma}_i,{\sigma}_j)=g(\tilde J E_i,\tilde J E_j)+g(E_i,E_j)+g(E_i,{\tilde J}^2E_j)+g(E_j,{\tilde J}^2E_i)=4g(E_i,E_j).$$
In particular, defining
$${\xi}_i:= \frac{1}{2}{\sigma}_i,$$
we have that $\{{\xi}_i\}_{1\leq i\leq n}$ is a $\hat g$-orthonormal frame field for $TM\oplus T^*M$.

\begin{proposition}
$$\delta \hat J=\frac{1}{4}[\delta J+\flat_g(\sum_{i=1}^n({\nabla}_{\tilde J{E}_i}J)E_i)],$$
where $\{E_i\}_{1\leq i\leq n}$ is a $g$-orthonormal frame field for $TM$.
\end{proposition}
\begin{proof} A direct computation gives:
$$\delta \hat J=-\sum_{i=1}^n({\hat \nabla}_{{\xi}_i}\hat J){\xi}_i=-\sum_{i=1}^n[{\hat \nabla}_{{\xi}_i}\hat J{\xi}_i-\hat J({\hat \nabla}_{{\xi}_i}{\xi}_i)]=$$
$$=-{1\over{4}}\sum_{i=1}^n[{\hat \nabla}_{\tilde J E_i}\hat J(\tilde J{E}_i+\flat_g(E_i))-\hat J({\hat \nabla}_{\tilde J E_i}(\tilde J{E}_i+\flat_g(E_i))]=$$
$$=-{1\over{4}}\sum_{i=1}^n[{\hat \nabla}_{\tilde J E_i}(J\tilde J{E}_i+\flat_g(\tilde J E_i)-J^*(\flat_g(E_i))+p\flat_g(E_i))
-\hat J(\nabla_{\tilde J E_i}\tilde J{E}_i+\nabla_{\tilde J E_i}\flat_g(E_i))]=$$
$$=-{1\over{4}}\sum_{i=1}^n[({\nabla}_{\tilde J E_i}J){\tilde J}{E}_i - \flat_g(({\nabla}_{\tilde J E_i}{ J}){E}_i)]=$$
$$={1\over{4}}[\delta J+\flat_g(\sum_{i=1}^n({\nabla}_{\tilde J E_i}J){E}_i)].$$
\end{proof}

\begin{lemma}
$$4(d\delta\hat J)(X+\alpha)=(d\delta J)(X)+ \flat_g(\sum_{i=1}^n{\nabla}_X(({\nabla}_{\tilde J E_i}{ J}){E}_i)),$$
for any $X\in C^{\infty}(TM)$ and $\alpha\in C^{\infty}(T^*M)$, where $\{E_i\}_{1\leq i\leq n}$ is a $g$-orthonormal frame field for $TM$.
\end{lemma}
\begin{proof} We have:
$$4(d\delta\hat J)(X+\alpha)=4{\hat \nabla}_{X+\alpha}\delta \hat J=4{\hat \nabla}_X\delta \hat J=
{\nabla}_{X}\delta J+\flat_g(\sum_{i=1}^n{\nabla}_X(({\nabla}_{\tilde J E_i}{ J}){E}_i))=$$
$$=(d\delta J)(X)+ \flat_g(\sum_{i=1}^n{\nabla}_X(({\nabla}_{\tilde J E_i}{ J}){E}_i)).$$
\end{proof}

\begin{lemma}
$$4(\delta d \hat J)(X+\alpha)=(\delta d J)(X)+ ({\nabla}^2 J^*)(\alpha)+$$
$$+\flat_g(\sum_{i=1}^n[-{\nabla}_{\tilde J E_i}(({\nabla}_X{ J}){E}_i)+({\nabla}_X J)({\nabla}_{\tilde J {E}_i}E_i)+({\nabla}_{{\nabla}_{\tilde J E_i}X} J)E_i]),$$
for any $X\in C^{\infty}(TM)$ and $\alpha\in C^{\infty}(T^*M)$, where ${\nabla}^2 J^*:=\sum_{i=1}^{n}(\nabla_{E_i}\nabla_{E_i}J^*-\nabla_{\nabla_{E_i}E_i}J^*)$ and $\{E_i\}_{1\leq i\leq n}$ is a $g$-orthonormal frame field for $TM$.
\end{lemma}
\begin{proof} We have:
$$4(\delta d\hat J)(X+\alpha)=-\sum_{i=1}^n({\hat \nabla}_{\tilde J E_i} d\hat J)(\tilde J E_i +{\flat}_g({E}_i),X+\alpha)=$$
$$=-\sum_{i=1}^n[{\hat \nabla}_{\tilde J E_i} (d\hat J)(\tilde J E_i +{\flat}_g({E}_i),X+\alpha)-(d\hat J)({\hat \nabla}_{\tilde J E_i}(\tilde J E_i +{\flat}_g({E}_i)),X+\alpha)-$$
$$-(d\hat J)(\tilde J E_i+{\flat}_g(E_i),{\hat \nabla}_{\tilde J E_i}(X+\alpha)]=$$
$$=-\sum_{i=1}^n[{\nabla}_{\tilde J E_i}( (d J)(\tilde J E_i,X))+{\nabla}_{\tilde J E_i}(({\nabla}_X J^*){\flat}_g({E}_i))-{\nabla}_{\tilde J E_i}(({\nabla}_{\tilde J E_i}J^*)\alpha)-$$
$$-(dJ)({\nabla}_{\tilde J E_i} \tilde J E_i,X)+({\nabla}_{{\nabla}_{\tilde J E_i}\tilde J E_i}J^*)\alpha-({\nabla}_X J^*)({\nabla}_{\tilde J E_i}\flat_g(E_i))-$$
$$-(dJ)(\tilde JE_i,{\nabla}_{\tilde J E_i}X)-{\nabla}_{{\nabla}_{\tilde JE_i}X}{\flat}_g(E_i)+({\nabla}_{\tilde J E_i}J^*)\alpha]=$$
$$=(\delta dJ)(X)+\sum_{i=1}^n[({\nabla}_{\tilde J E_i}{\nabla}_{\tilde J E_i}J^*)\alpha-(\nabla_{\nabla_{\tilde JE_i}\tilde JE_i}J^*)\alpha]+$$
$$+\flat_g(\sum_{i=1}^{n}[-{\nabla}_{\tilde J E_i}((\nabla_X J)E_i)+({\nabla}_{\nabla_{\tilde J E_i} X}J)E_i+(\nabla_XJ)(\nabla_{{\tilde J E_i}}E_i)])=$$
$$=(\delta d J)(X)+ ({\nabla}^2 J^*)(\alpha)+$$
$$+\flat_g(\sum_{i=1}^n[-{\nabla}_{\tilde J E_i}(({\nabla}_X{ J}){E}_i)+({\nabla}_X J)({\nabla}_{\tilde J {E}_i}E_i)+({\nabla}_{{\nabla}_{\tilde J E_i}X} J)E_i]).$$
\end{proof}

Defining
$$(\Delta J^*)(\alpha):={\flat}_g((\Delta J)({\sharp}_g\alpha)),$$
we have the following Weitzenb\"{o}ck formula for $J^*$:

\begin{lemma}
$$(\Delta J^*)(\alpha)=-({\nabla}^2 J^*)(\alpha)-\flat_g(\sum_{i=1}^{n}(R(E_i,{\sharp}_ g \alpha)J)E_i),$$
for any $\alpha\in C^{\infty}(T^*M)$, where $\{E_i\}_{1\leq i\leq n}$ is a $g$-orthonormal frame field for $TM$.
\end{lemma}
\begin{proof}
It follows immediately from the Weitzenb\"{o}ck formula for $J$ \cite{xi}.
\end{proof}

Then the expression of the Hodge-Laplace operator on $\Gamma(\Lambda^p(TM\oplus T^*M)^*\otimes (TM\oplus T^*M))$ computed on $\hat{J}$ is given by:

\begin{proposition}
Let $(M,J,g)$ be an $n$-dimensional metallic Riemannian manifold such that $J^2=pJ+qI$ with $p^2+4q>0$. Then for any $X\in C^{\infty}(TM)$ and $\alpha\in C^{\infty}(T^*M)$:
$$4(\Delta \hat J)(X+\alpha)=(\Delta J) (X)-(\Delta J^*)(\alpha)+$$
$$+\flat_g(\sum_{i=1}^{n}[(R(X,\tilde J E_i)J)E_i-(R(E_i,{\sharp}_ g \alpha)J)E_i+({\nabla}_{\tilde J E_i}J)({\nabla}_X{E_i})]),$$
where $\{E_i\}_{1\leq i\leq n}$ is a $g$-orthonormal frame field for $TM$.
\end{proposition}
\begin{proof} Let us consider a $g$-orthonormal frame field $\{F_i\}_{1\leq i\leq n}$ in a neighborhood of a point $x\in M$ such that $({{\nabla}_{F_i} F_j})(x)=0$. Remark that $E_i:=\tilde J F_i$ defines a $g$-orthonormal frame field and $\tilde J E_i={\tilde J}^2 F_i=F_i$. Then $({\nabla}_{\tilde J E_i} {\tilde J }E_j)(x)=0$.

From the previous computations we get:
$$4(\Delta \hat J)(X+\alpha)=(d\delta J)(X)+ \flat_g(\sum_{i=1}^n{\nabla}_X(({\nabla}_{\tilde J E_i}{ J}){E}_i)+(\delta d J)(X)+ ({\nabla}^2 J^*)(\alpha)+$$
$$+\flat_g(\sum_{i=1}^n[-{\nabla}_{\tilde J E_i}(({\nabla}_X{ J}){E}_i)+({\nabla}_X J)({\nabla}_{\tilde J {E}_i}E_i)+({\nabla}_{{\nabla}_{\tilde J E_i}X} J)E_i])=$$
$$=(\Delta J) (X)-(\Delta J^*)(\alpha)+$$
$$+\flat_g(\sum_{i=1}^{n}[(R(X,\tilde J E_i)J)E_i-(R(E_i,{\sharp}_ g \alpha)J)E_i+({\nabla}_{\tilde J E_i}J)({\nabla}_X{E_i})]).$$
\end{proof}

\begin{corollary}
$\hat J$ is harmonic if and only if the following conditions hold:
\begin{enumerate}
  \item $J$ is harmonic;
  \item $\sum_{i=1}^{n} (R(E_i,X)J)E_i=0$;
  \item $\sum_{i=1}^{n} [(R(X,\tilde J E_i)J)E_i+({\nabla}_{\tilde J E_i}J)({\nabla}_X{E_i})]=0$,
\end{enumerate}  
for any $X \in C^{\infty}(TM)$, where $\{E_i\}_{1\leq i\leq n}$ is a $g$-orthonormal frame field for $TM$.
\end{corollary}

\section{Harmonic maps between generalized tangent bundles and harmonic generalized metallic structures}

Let $(M,J,g)$ and $(\bar{M},\bar{g},\bar{J})$ be two $n$-dimensional metallic Riemannian manifolds. Denote by $\nabla$ and respectively, $\bar{\nabla}$ the Levi-Civita connections associated to $g$ and respectively, $\bar{g}$.

Consider $\Phi:(M,J,g)\rightarrow (\bar{M},\bar{J},\bar{g})$ a diffeomorphism and define:
$$\hat{\Phi}: TM\oplus T^*M \rightarrow T\bar{M}\oplus T^*\bar{M}, \ \ \hat{\Phi}(X+\alpha):=\Phi_*X+(\Phi^*)^{-1}\alpha,$$
for any $X\in C^{\infty}(TM)$ and $\alpha\in C^{\infty}(T^*M)$.

\begin{remark}
If $\Phi:(M,J,g)\rightarrow (\bar{M},\bar{J},\bar{g})$ is a metallic map, then
$$\hat{\bar{J}}\circ\hat{\Phi}=\hat{\Phi}\circ \hat{J},$$
where $(\hat J, \hat g)$ and $(\hat{\bar{J}}, \hat{\bar{g}})$ are the generalized metallic Riemannian structures induced by the Riemannian structures $(J,g)$ and $(\bar{J},\bar{g})$ \cite{bn}.
\end{remark}

The tension field of $\hat{\Phi}$ is defined by
$$\tau(\hat{\Phi}):=\sum_{i=1}^{n}[\hat{\bar{\nabla}}_{\hat{\Phi}_*\xi_i}\hat{\Phi}_*\xi_i-\hat{\Phi}_*(\hat{\nabla}_{\xi_i}\xi_i)],$$
where $\{\xi_i\}_{1\leq i\leq n}$ is a $\hat{g}$-orthonormal frame field
on $TM\oplus T^*M$.

\begin{proposition}\label{p999}
Let $\Phi:(M,J,g)\rightarrow (\bar{M},\bar{J},\bar{g})$ be an isometry. Then $\hat{\Phi}$ is a harmonic map if and only if $\Phi$ is a harmonic map and
$$\sum_{i=1}^{n}[\bar{\nabla}_{\Phi_*(JE_i)}(\Phi_*E_i)-\Phi_*(\nabla_{JE_i}E_i)]=0,$$
for $\{E_i\}_{1\leq i\leq n}$ a $g$-orthonormal frame field
on $TM$.
\end{proposition}
\begin{proof}
Let $\xi_i:=\frac{1}{2}\tilde{J}E_i+\frac{1}{2}\flat_g(E_i)$ be a $\hat{g}$-orthonormal frame field
on $TM\oplus T^*M$ for $\tilde{J}:={-1\over{\sqrt{p^2+4q}}}(2J-pI)$ and $\{E_i\}_{1\leq i\leq n}$ a $g$-orthonormal frame field
on $TM$.
A direct computation gives:
$$\tau(\hat{\Phi})=\frac{1}{4}\sum_{i=1}^{n}[\bar{\nabla}_{\Phi_*(\tilde{J}E_i)}(\Phi_*(\tilde{J}E_i))-\Phi_*(\nabla_{\tilde{J}E_i}\tilde{J}E_i)+$$$$+
\bar{\nabla}_{\Phi_*(\tilde{J}E_i)}(\Phi^*)^{-1}(\flat_g(E_i))-(\Phi^*)^{-1}(\nabla_{\tilde{J}E_i}\flat_g(E_i))]:=$$$$
:=\frac{1}{4}\tau(\Phi)+\frac{1}{4}\sum_{i=1}^{n}
[\bar{\nabla}_{\Phi_*(\tilde{J}E_i)}\flat_{\bar{g}}(\Phi_*E_i)-(\Phi^*)^{-1}\flat_g(\nabla_{\tilde{J}E_i}E_i)]=$$$$
=\frac{1}{4}\tau(\Phi)+\frac{1}{4}\sum_{i=1}^{n}
[\flat_{\bar{g}}(\bar{\nabla}_{\Phi_*(\tilde{J}E_i)}(\Phi_*E_i))-\flat_{\bar{g}}(\Phi_*(\nabla_{\tilde{J}E_i}E_i))]=
$$$$=\frac{1}{4}\tau(\Phi)+\frac{1}{4}\flat_{\bar{g}}(\sum_{i=1}^{n}[\bar{\nabla}_{\Phi_*(\tilde{J}E_i)}(\Phi_*E_i)-\Phi_*(\nabla_{\tilde{J}E_i}E_i)])=$$$$
=\frac{1}{4}\tau(\Phi)+\frac{p}{4\sqrt{p^2+4q}}\flat_{\bar{g}}(\tau(\Phi))-
\frac{1}{2\sqrt{p^2+4q}}\flat_{\bar{g}}(\sum_{i=1}^{n}[\bar{\nabla}_{\Phi_*(JE_i)}(\Phi_*E_i)-\Phi_*(\nabla_{JE_i}E_i)]).$$
\end{proof}

\begin{proposition}
Let $\Phi:(M,J,g)\rightarrow (\bar{M},\bar{J},\bar{g})$ be a metallic isometry and let $(\hat J, \hat g)$ and $(\hat{\bar{J}}, \hat{\bar{g}})$ be the generalized metallic Riemannian structures induced by the non trivial metallic Riemannian structures $(J,g)$ and $(\bar{J},\bar{g})$ \cite{bn}. Then
$$\hat{\bar{J}}(\tau(\hat{\Phi}))=-\frac{1}{4}[\Phi_*(\delta J)-\delta \bar{J}]+\frac{1}{4}\sum_{i=1}^{n}[\bar{\nabla}_{\Phi_*E_i}\Phi_*(JE_i)-\Phi_*(\nabla_{E_i}JE_i)]+$$$$+
\frac{p^2+\sqrt{p^2+4q}}{4\sqrt{p^2+4q}}\flat_{\bar{g}}(\tau(\Phi))+\frac{p}{4\sqrt{p^2+4q}}\flat_{\bar{g}}(\Phi_*(\delta J)-\delta \bar{J})
-$$$$-\frac{p}{4\sqrt{p^2+4q}}\flat_{\bar{g}}(\sum_{i=1}^{n}[\bar{\nabla}_{\Phi_*E_i}\Phi_*(JE_i)-\Phi_*(\nabla_{E_i}JE_i)])
-$$$$-\frac{p}{2\sqrt{p^2+4q}}\flat_{\bar{g}}(\sum_{i=1}^{n}[\bar{\nabla}_{\Phi_*(JE_i)}\Phi_*E_i-\Phi_*(\nabla_{JE_i}E_i)])
+$$$$+\frac{1}{2\sqrt{p^2+4q}}\flat_{\bar{g}}(\bar{J}(\sum_{i=1}^{n}[\bar{\nabla}_{\Phi_*(JE_i)}\Phi_*E_i-\Phi_*(\nabla_{JE_i}E_i)])),
$$
for $\{E_i\}_{1\leq i\leq n}$ a $g$-orthonormal frame field
on $TM$.
\end{proposition}
\begin{proof}
It follows from Propositions \ref{p} and \ref{p999}.
\end{proof}

\textit{Adara M. Blaga}

\textit{Department of Mathematics}

\textit{West University of Timi\c{s}oara}

\textit{Bld. V. P\^{a}rvan nr. 4, 300223, Timi\c{s}oara, Rom\^{a}nia}

\textit{adarablaga@yahoo.com}

\bigskip

\textit{Antonella Nannicini}

\textit{Department of Mathematics and Informatics "U. Dini"}

\textit{University of Florence}

\textit{Viale Morgagni, 67/a, 50134, Firenze, Italy}

\textit{antonella.nannicini@unifi.it}


\begin{thebibliography}{99}

\bibitem{bej} C.-L. Bejan and M. Benyounes, {\it Harmonic Maps Between Almost Para-Hermitian Manifolds},
New Developments in Diff. Geom., Budapest, 67-76, 1996.

\bibitem{ablaga} A. M. Blaga, {\it Affine connections on almost para-cosymplectic manifolds},
Czechoslovak Math. J. {\bf 61} (2011), no. 3, 863-871.

\bibitem{ada} A. M. Blaga, {\it Harmonic subtangent structures},
Journal of Mathematics, Vol. 2014, Article ID 603078.

\bibitem{bn} A. M. Blaga and A. Nannicini, {\it Generalized metallic pseudo-Riemannian structures}, arXiv:1811.10406v1.

\bibitem{bn1} A. M. Blaga and A. Nannicini, {\it On curvature tensors of Norden and metallic pseudo-Riemannian manifolds}, Complex manifolds 2019; 6:150-159.

\bibitem{bn2} A. M. Blaga and A. Nannicini, {\it On the geometry of metallic pseudo-Riemannian structures}, to appear in Riv. Mat. Univ. Parma.

\bibitem{fe} D. Fetcu, {\it Harmonic maps between complex Sasakian manifolds},
Rend. Sem. Mat. Univ. Pol. Torino {\bf 64} (2006), no. 3, 319-329.

\bibitem{fo} A. Fotiadis, {\it Harmonic Maps Between Noncompact Manifolds},
J. Nonlinear Math. Phys. {\bf 15} (2008), no. 3, 176-184.

\bibitem{i} S. Ianu\c s, R. Mazzocco and G. E. V\^ ilcu, {\it Harmonic maps between quaternionic K\"{a}hler manifolds}, J. Nonlinear Math. Phys. {\bf 15} (2008), no. 1, 1-8.

\bibitem{pa} S. Ianu\c s and A. M. Pastore, {\it Harmonic maps on contact metric manifolds}, An. Math. Blaise Pascal {\bf 2} (1995), no. 2, 43-53.

\bibitem{ja} J. P. Jaiswal, {\it Harmonic maps on Sasakian manifolds},
J. Geom. \textbf{104} (2013), no. 2, 309-315.

\bibitem{ji} W. Jianming, {\it Harmonic complex structures}, (Chinese) Chinese Ann. Math. Ser. A \textbf{30} (2009), no. 6, 761-764.

\bibitem{li} J. Li, {\it Stable harmonic maps between Finsler manifolds and SSU manifolds},
Commun. Contemp. Math. {\bf 14} (2012), no. 3.

\bibitem{sa} B. Sahin, {\it Harmonic Riemannian maps on locally conformal K\"{a}hler manifolds}, Proc. Indian Acad. Sci. (Math. Sci.) {\bf 118} (2008), no. 4, 573-581.

\bibitem{xi} Y. L. Xin, {\it Geometry of Harmonic Maps}, Progress in Nonlinear Diff. Eq. and their Appl. {\bf 23}, Birkh\"{a}user, 1996.

\end{thebibliography}
\end{document}